\documentclass[reqno,12pt]{amsart}
\usepackage{amsfonts,amssymb}
\usepackage{enumerate}
\usepackage{mathrsfs}
\usepackage{verbatim}
\usepackage{url}

\newcommand{\R}{\mathbb{R}}
\newcommand{\Q}{\mathbb{Q}}
\newcommand{\Z}{\mathbb{Z}}
\newcommand{\N}{\mathbb{N}}
\newcommand{\su}{\mathcal{S}}

\newcommand{\e}{\varepsilon}
\newcommand{\de}{\delta}
\newcommand{\M}{\mathcal{M}}
\DeclareRobustCommand{\rchi}{{\mathpalette\irchi\relax}}
\newcommand{\irchi}[2]{\raisebox{\depth}{$#1\chi$}}

\newtheorem{thm}{Theorem}%[section]
\newtheorem{lem}{Lemma}

\theoremstyle{definition}

\newtheorem{conj}{Conjecture}
\newtheorem*{prob}{Main Problem}
\theoremstyle{remark}
\newtheorem{rem}{Remark}
\newtheorem{ex}{Example}

\numberwithin{equation}{section}

\author{Jing-Jing Huang}
\address{
Jing-Jing Huang: Department of Mathematics and Statistics, University of Nevada, Reno,
1664 N. Virginia St., Reno, NV 89557}
\email{jingjingh@unr.edu}
\dedicatory{Dedicated to Professor Robert C. Vaughan on the occasion of his 70th birthday}
\thanks{Research is supported by the UNR VPRI startup grant 1201-121-2479}

\begin{document}

\title
[rational points near hypersurfaces]
{The density of rational points near hypersurfaces}

\begin{abstract}
We establish a sharp asymptotic formula for the number of rational points up to a given height and within a given distance from a hypersurface. Our main innovation is a bootstrap method that relies on the synthesis of Poisson summation, projective duality and the method of stationary phase. This has surprising applications to counting rational points lying on the manifold; indeed, we are able to prove an analogue of Serre's Dimension Growth Conjecture (originally stated for projective varieties) in this general setup.  As another consequence of our main counting result, we solve the generalized Baker-Schmidt problem in the simultaneous setting for hypersurfaces. 
\end{abstract}
\maketitle

\section{The Distribution of Rational Points near  Manifolds}\label{s1}
A fundamental theme in mathematics is the study of integral and rational points
on algebraic varieties. A great variety of  analytic, algebraic, geometric and arithmetic tools have been introduced into play in this fast growing field, which has a beautiful name, \emph{diophantine geometry}, after Serge Lang published a book with the same title in 1962 \cite{la}. It turns out that one can approach this problem from a broader perspective, namely studying the integral/rational points lying `close' to a manifold.

Let $\mathcal{M}$ be a bounded submanifold of $\mathbb{R}^n$ of dimension $m$. Given $Q\ge2$ and $\delta\in(0,1/2)$, let
\begin{equation}\label{e1.1}N_{\mathcal{M}}(Q,\delta):=\#\left\{{\mathbf{p}}/q\in\mathbb{Q}^n: 1\le q\le Q, \mathrm{dist}_\infty\left({\mathbf{p}}/q,\mathcal{M}\right)\le{\delta}/q\right\},
\end{equation}
where $\mathbf{p}\in\mathbb{Z}^n, q\in\mathbb{Z}$ and  $\mathrm{dist}_\infty(\cdot,\cdot)$ is the $L^\infty$ distance on $\mathbb{R}^n$. The following intricate problem presents a major challenge.

\begin{prob}
Estimate $N_{\mathcal{M}}(Q,\delta)$ for a `generic' manifold $\mathcal{M}$.
\end{prob}

Though the distribution of rational points near manifolds is a very interesting problem in its own right, progress on the Main Problem also has many significant consequences for other problems in number theory.  See \S  \ref{s2} for an application to a diophantine inequality, \S\ref{s3} for an application to the dimension growth conjecture and \S \ref{s4} for an application to metric diophantine approximation on manifolds. For other applications, see references \cite{Elk,HR,Maz,Sch}. 

We start by observing the following trivial bound:
\begin{equation}\label{e1.2}
 N_{\mathcal{M}}(Q,\delta)\ll Q^{m+1}.
\end{equation}
Throughout this paper, we will use the notation $\ll$ and $O$ freely. The implicit constants will depend on the underlying manifold and/or some fixed parameters say $\e$, but they will never depend on varying parameters $Q$ and $\de$.
\begin{ex}\label{ex1}
If $\M$ contains a piece of a rational $m$-linear subspace of $\R^n$, then
$$
N_{\mathcal{M}}(Q,\delta)\asymp Q^{m+1}.
$$
\end{ex}

Example \ref{ex1} shows that the trivial bound \eqref{e1.2} is the sharpest general bound for all manifolds. So the prospect of establishing nontrivial bounds has to rely on the elimination of \emph{flatness} of the manifold,
which means the manifold $\M$ under question has to be \emph{properly curved}.

Let $k=n-m$ be the codimension of $\M$. Based on a simple probabilistic argument  one should expect the following heuristic bound
\begin{equation}\label{e1.0}
N_{\mathcal{M}}(Q,\delta)\asymp \delta^k Q^{m+1}.
\end{equation}
Explicit examples suggest that this heuristic does not hold unconditionally, unless one imposes a reasonable lower bound on $\delta$ in terms of $Q$. Indeed, the manifold $\M$ may contain many rational points.

\begin{ex}\label{ex2}
Consider the parabola $\mathcal{P}$: $y=x^2$, $x\in[0,1]$. Clearly $(a/q,a^2/q^2)\in \mathcal{P}$ with $1\le a\le q$ and $q\le\sqrt{Q}$. So
$N_\mathcal{P}(Q,\delta)\ge N_\mathcal{P}(Q,0) \ge \sum_{q\le \sqrt{Q}}q\gg Q$.
\end{ex}

Example \ref{ex2} shows that when $Q$ is fixed the heuristic \eqref{e1.0} does not hold for arbitrarily small $\delta$. So there has to be an error term which does not go to zero as $\delta\to 0$ or one has to insist that the $\delta$ not be smaller than a certain threshold in order for the heuristic main term $\delta^k Q^{m+1}$ to be dominant. The following example indicates generally where that threshold is in terms of $Q$. 

\begin{ex}\label{ex3}
 Consider the hypersphere $\mathcal{S}=S^{n-1}$ in $\R^n$ $(n\ge2)$ defined by the equation $$x_1^2+x_2^2+\cdots+x_n^2=1.$$ 
As a standard application of the Hardy-Littlewood circle method, one can show that the asymptotic formula $$N_{S^{n-1}}(Q,0)\sim c_{n}Q^{n-1}$$
holds as $Q\rightarrow\infty$, where $c_n$ is a positive constant which can be written as an infinite product of local densities. 
\end{ex}

It is readily verified that the heuristic \eqref{e1.0} holds for Example \ref{ex3} only if $\delta Q^{1/k}\to\infty$. 
Therefore, it is reasonable to propose the following plausible conjecture.

\begin{conj}\label{c0}[The Main Conjecture] For any compact submanifold $\mathcal{M}$ of $\mathbb{R}^n$ with proper curvature conditions, we have
$$N_\mathcal{M}(Q,\delta)\sim c_\M \delta^kQ^{m+1}$$
when $\de\ge Q^{-\frac1k+\varepsilon}$ for some $\varepsilon>0$ and $Q\to \infty$. 
\end{conj}
Here and throughout the paper, a compact submanifold of $\mathbb{R}^n$ means that it is a bounded immersed submanifold of $\mathbb{R}^n$ with boundary. 
Conjecture \ref{c0} implies that 
\begin{equation}\label{e1.3}
N_\mathcal{M}(Q,\delta)\ll \delta^kQ^{m+1}+Q^{m+\e}\quad\text{for all }\de\text{ and }Q
\end{equation}
by the monotonicity of $N_{\M}(Q,\delta)$ in $\de$,
and that
\begin{equation}\label{e1.4}
N_\mathcal{M}(Q,\delta)\gg \delta^kQ^{m+1}\quad\text{ for }\delta\ge Q^{-\frac1k+\e}\text{ and large }Q.
\end{equation}
Conjecture \ref{c0} will be hitherto referred to as \emph{the Main Conjecture}, which is the primary object of interest in this paper.

We emphasize again that a \emph{rigid curvature condition} is absolutely necessary in the Main Conjecture, since a thin neighborhood of a flat point in a manifold could contain too many rational points, cf. Example \ref{ex1}. We will elaborate more on these curvature conditions later.
\begin{ex}\label{ex4}
Consider the Fermat curve $\mathcal{F}_n$
$$
x^n+y^n=1.
$$
$\mathcal{F}_n$ has order of contact $n-1$ with the lines $x=1$ and $y=1$. Elementary calculation shows that
$$
N_{\mathcal{F}_n}(Q,\de)\gg \delta^{\frac1 n}Q^{2-\frac1n}.
$$
This means that the bound
$$
N_{\mathcal{F}_n}(Q,\de)\ll \de Q^2+Q^{1+\e}
$$
cannot hold, when
$$
\de\gg Q^{-\frac1{n-1}+\e}.
$$
\end{ex}
Although the Fermat curve $\mathcal{F}_n$ is not globally flat, it has a few \emph{virtually flat} points. Example \ref{ex4} shows that even when the curvature only vanishes at one point, the Main Conjecture could break down.

 Though a great deal is known about sharp upper and lower bounds for $N_{\mathcal{M}}(Q,\delta)$, The Main Problem  yet remains wide open in general. Below we will briefly survey the history of this problem in the literature; in this way we put our main result in context.

\subsection{Planar curves} Let $\mathcal{C}$ be a compact curve in $\mathbb{R}^2$, with curvature bounded away from zero. This is the first nontrivial case that has been studied extensively in the literature. By the implicit function theorem,  we can assume, without loss of generality, that $\mathcal{C}$ is given by $f: I\rightarrow\R$ with some finite interval $I$. Huxley \cite{hux} was the first to obtain a near optimal upper bound for $C^2$ curves, 
\begin{equation}\label{e1.5}
N_{\mathcal{C}}(Q,\delta)\ll_{\mathcal{C}} \delta^{1-\varepsilon} Q^2+Q\log Q \quad\text{ for any  }\delta\text{ and } Q.\footnote{The version we cite here is actually \cite[Lemma 2.2]{VV}, which is a direct consequence of Huxley's main result. Huxley counts rational points uniquely, while we count with multiplicities. It is however a routine exercise to transform one version into the other using M\"{o}bius transform/inversion, cf. Remark \ref{remark3}.}
\end{equation}
which comes within a whisker of proving the predicted upper bound \eqref{e1.3} for planar curves. The latter was remarkably established by Vaughan and Velani \cite{VV}, who showed that
\begin{equation}\label{e1.7}
N_{\mathcal{C}}(Q,\delta)\ll_{\mathcal{C}} \delta Q^2+Q^{1+\e}
\end{equation}
for  $C^3$ curves.

On the other hand, a sharp lower bound, established by Beresnevich, Dickinson and Velani \cite{BDV}, says 
$$N_{\mathcal{C}}(Q,\delta)\gg_{\mathcal{C}} \delta Q^2\quad\text{ for }\delta\gg Q^{-1}\text{ and }\delta Q\to\infty$$ for $C^3$ curves $\mathcal{C}$ having nonvanishing curvature at at least one point. 

These two undoubtedly fascinating results give satisfactory answers to the Main Problem by establishing the corresponding upper and lower bounds \eqref{e1.3}, \eqref{e1.4} respectively in the case of planar curves. However, the corresponding asymptotic formula had not been established until recently the author  \cite{hua1} proved  that

\begin{equation}\label{e1.8}
{N}_\mathcal{C}(Q,\delta)=|I|\delta Q^2+O\left(\delta^{\frac12}\left(\log \delta^{-1}\right)Q^{\frac32}+Q^{1+\varepsilon}\right)
\end{equation}
for $C^3$ curves, which solves the Main Conjecture in this case.

The proof of \eqref{e1.8} relies on the work of Huxley \cite{hux} and Vaughan-Velani \cite{VV} and in fact the harmonic analysis turns the upper bound \eqref{e1.5} into an asymptotic formula. We refer interested readers to \cite{hua1} for more details and discussions about the strategy behind the proof of \eqref{e1.8}. The primary goal of the current memoir is to investigate the Main Problem for manifolds of higher dimensions, especially hypersurfaces. 

\subsection{General manifolds}\label{s1.2} Higher dimensions present many more challenges as well as many more opportunities.  In the spectacular work \cite{bere}, Beresnevich established the sharp lower bound \eqref{e1.4} for any analytic submanifold $\mathcal{M}$ of $\mathbb{R}^n$ which is nondegenerate at one point,  namely 
$$N_{\mathcal{M}}(Q,\delta)\gg_{\M} \delta^k Q^{m+1}\quad\text{ for any } \delta\gg Q^{-1/k}.$$ 
A manifold $\mathcal{M}$ is \emph{nondegenerate} at $x\in\mathcal{M}$ if $\mathcal{M}$ has at most finite order of contact with any hyperplane that passes through ${x}$.

The complementary upper bound \eqref{e1.3} as well as the Main Conjecture still remains a major challenging question. A quite intriguing feature is that nondegeneracy alone is not enough to reverse the above inequality sign, cf. Example \ref{ex4}; a rigid curvature bound has to be imposed.  

A straightforward application of exponential sum techniques, Poisson summation and optimal pointwise bounds on oscillatory integrals yields a bound which is far from the conjectured upper bound \eqref{e1.3}. This is partly due to the fact that, when the distance $\delta$ is very small, the resulting exponential sum is quite long, which causes an enormous amount of loss in the estimate if we use a pointwise bound on each resulting oscillatory integral. Currently there are a couple of approaches to circumvent this obstacle. Huxley's elementary method \cite{hux} represents an ingenious implementation of Swinnerton-Dyer's determinant method \cite{SD}; this unfortunately seems to only pertain to plane geometry and is therefore unlikely to entail generalizations to higher dimensions. If insisting on using the fourier analytic approach, one can enhance the pointwise bound provided that the manifold is locally at least 2-convex, which only holds generically for manifolds with very small codimension and completely rules out curves \cite{BeD}. Another possibility is a linearization argument found by Sprindz\v{u}k \cite{Spr}, where he constructs small linear patches to approximate the manifold and hence reduces the length of the exponential sum. Beresnevich, Vaughan, Velani and Zorin  \cite{BVVZ} explored this idea further and established, among other things, that, in the case when the manifold is a hypersurface $\mathcal{S}$ with Gaussian curvature bounded away from zero,
\begin{equation}\label{e1.6}
 N_\mathcal{S}(Q,\delta)\ll \delta Q^{n}+Q^{n-1+\frac2{n+1}}(\log Q)^{\frac{2n-2}{n+1}}.
\end{equation}
Here the first term $\delta Q^n$ matches with Conjecture \ref{c0}, whereas the second term still misses the conjectured upper bound \eqref{e1.3} by roughly $Q^{\frac2{n+1}}$.

We are able to sharpen this further and obtain the conjectured upper bound \eqref{e1.3}. In fact, Theorem \ref{t2} below even refines \eqref{e1.3}. Let 
\begin{equation}\label{e1.16}
l=\max\left(\left\lfloor\frac{n-1}2\right\rfloor+5,n+1\right),
\end{equation}
\begin{equation}\label{e1.14}
E_3(Q)=Q^{2}\exp(c\sqrt{\log Q})
\end{equation}
and
\begin{equation}\label{e1.15}
E_n(Q)=Q^{n-1}(\log Q)^\kappa,\quad n\ge4
\end{equation}
for some positive constants $c$ and $\kappa$ that can be effectively calculated from our proof.

\begin{thm}\label{t2}
Let $\mathcal{S}$ be a compact $C^l$ hypersurface in $\mathbb{R}^n$ with Gaussian curvature bounded away from zero. Then
 
$$N_\mathcal{S}(Q,\delta)\ll \delta Q^{n}+E_n(Q), $$
where $c, \kappa$ and the implied constant only depend on $\mathcal{S}$.
\end{thm}

By compactness, $\mathcal{S}$ may be covered by finitely many local charts. So, without loss of generality, we may assume that the hypersurface $\mathcal{S}$ is given in the Monge form 
\begin{equation}\label{e1.9}
(\mathbf{x},f(\mathbf{x})),\quad \mathbf{x}=(x_1,x_2,\ldots,x_{n-1})\in\mathcal{D}, f\in C^{l}(\mathcal{D})
\end{equation}
where $\mathcal{D}\subseteq\R^{n-1}$ is a connected bounded open set.  The curvature condition in Theorem \ref{t2} then guarantees that for all $\mathbf{x}\in\mathcal{D}$ the Hessian $\nabla^2f(\mathbf{x}):=\{\frac{\partial^2{f}}{\partial x_i\partial x_j}\}$ satisfies 
\begin{equation}\label{e1.10}
0<c_1\le|\det \nabla^2f(\mathbf{x})|\le c_2
\end{equation}
for some positive numbers $c_1$ and $c_2$.

Under the above Hessian condition, we know the gradient $\nabla f:=\left(\frac{\partial f}{\partial x_1},\frac{\partial f}{\partial x_2},\ldots,\frac{\partial f}{\partial x_{n-1}}\right)$ is a diffeomorphism in a sufficiently small neighborhood of any $\mathbf{x}\in \mathcal{D}$ by the Inverse Function Theorem. So we may take $\mathcal{D}$ small enough such that 
\begin{equation}\label{e1.12}
\nabla f: \mathcal{D}\rightarrow \nabla f(\mathcal{D})\textrm{ is a diffeomorphism}.
\end{equation}

For the method in this paper, it is more convenient to smooth the counting function with a weight function $w\in C_0^\infty(\R^{n-1})$ with $\mathrm{supp}(w)\subseteq\mathcal{D}$. Let
\begin{equation}\label{e1.11}
N_{\su}^w(Q,\delta):=\sum_{\substack{\mathbf{a}\in\Z^{n-1}\\q\le Q\\ \|qf(\mathbf{a}/q)\|<\delta}}w\left(\frac{\mathbf{a}}{q}\right).
\end{equation}

\begin{thm}\label{t1}
Let $\mathcal{S}$ be a hypersurface given by \eqref{e1.9} and satisfy \eqref{e1.10} and \eqref{e1.12}. Then
 we have
$$
N_\mathcal{S}^w(Q,\delta)=\frac{2\hat{w}(0)}{n} \delta Q^n+O_{\mathcal{S},w}(E_n(Q)),
$$
where the constants $c$ and $\kappa$ only depend on $\mathcal{S}$ and $w$.

\end{thm}

\begin{rem}
Notice that our method still works in principle when $n=2$. The argument in the proof of Theorem \ref{t2} should yield
\begin{equation}\label{e1.13}
N_{\mathcal{C}}(Q,\delta)\ll \de Q^2+Q^{3/2}(\log Q)^\kappa
\end{equation}
for a planar curve $\mathcal{C}$ with some constant $\kappa>0$. We may compare \eqref{e1.13}  with the bounds \eqref{e1.5} and \eqref{e1.7}. Incidentally, it was also obtained implicitly by W.M. Schmidt \cite{Sc1} with the restriction $\delta\in(Q^{-2},Q^{-1})$ in order to prove the celebrated result on the extremality of planar curves \cite{Sc2}. 
\end{rem}

\begin{rem}
We clearly see from Example \ref{ex3} that, except for the logarithmic factors, the bound in Theorem \ref{t2} and the error term in Theorem \ref{t1} are sharp. 
\end{rem}

With the help of Theorem \ref{t1}, we are able to establish the Main Conjecture for hypersurfaces with nonvanishing Gaussian curvature.

\begin{thm}\label{t3}
 Let $\mathcal{S}$ be a hypersurface given by \eqref{e1.9} and satisfy \eqref{e1.10}, and $K\subseteq\mathcal{D}$ be convex and compact. Then
 we have for $\delta>Q^{-1+\varepsilon'}$ with any fixed $\varepsilon'>0$ 
 $$
N_\mathcal{S}^K(Q,\delta)=(1+o(1))\frac{2|K|}{n} \delta Q^n,
$$
where 
$$
N_{\su}^K(Q,\delta):=\sum_{\substack{\mathbf{a}\in\Z^{n-1}\\q\le Q\\ \|qf(\mathbf{a}/q)\|<\delta}}\rchi_K\left(\frac{\mathbf{a}}{q}\right)
$$
with $\rchi_K$ the characteristic function of $K$
and $|K|$ the Lebesgue measure of $K$.

\end{thm}

Theorem \ref{t3} is the first time that the Main Conjecture has been established beyond the case of planar curves \cite{hua1}.

\begin{rem}\label{remark3}
Since $\mathbf{a}$ and $q$ are not required to be coprime in \eqref{e1.1}, we actually count rational points with multiplicities. It is, however, a routine application of the M\"{o}bius inversion formula to derive new versions of Theorem \ref{t2}, \ref{t1} and \ref{t3} with the coprime condition, from the current ones. This has been done in the case of planar curves \cite[Theorem 1]{hua1} and the same argument works for the hypersurfaces dealt with in this paper.
\end{rem}

\begin{rem}
It is worth noting that in Theorem \ref{t3} the condition \eqref{e1.12} is not assumed. But, compared with Theorem \ref{t1}, the trade off is the quantitative error term. 
\end{rem}

Our strategy to prove Theorem \ref{t2} is distinctly different from all the previous approaches outlined in the beginning of \S\ref{s1.2}. The  argument starts with delicate harmonic analysis and an elegant duality principle, which is a higher dimensional  analogue of the $B$-process in the van der Corput method for exponential sums. In the case of planar curves, this duality (which corresponds to the usual one dimensional van der Corput $B$-process) has been used in \cite{VV, hua1}. However, Huxley's bound \eqref{e1.5} is crucial in \cite{VV, hua1}, which serves as the input on the dual side. To the best of our knowledge, Huxley's method is only pertinent to the plane geometry and does not admit higher dimensional generalization, which represents the major difficulty in generalizing the method in \cite{hua1} to higher dimensions. The main innovation of this paper is to use a monotonicity argument pertaining to our counting problem to bootstrap the bounds coming from the multidimensional Poisson summation formula and the method of stationary phase. Indeed, we reproduce the bound \eqref{e1.6} (which is established in \cite{BVVZ} by a very different method) by the first application of this process. Further iteration of the argument yields the best possible bound in Theorem \ref{t2}. The details of the proof are presented in \S \ref{s7}.

\begin{rem}Since the van der Corput method is mentioned here, an alert reader may wonder what happens if one applies the method in the usual way (successively applying the $A$- and $B$-processes). This approach has been successfully used to attack other related problems, e.g. the dimension growth conjecture for hypersurfaces with degree at least four \cite{Mar}. However, it turns out that this method is disadvantageous in our problem. Indeed, in the planar case, even assuming the far reaching Exponent Pair Conjecture, one can merely reproduce the bound \eqref{e1.13}, which is certainly worse than \eqref{e1.7} and \eqref{e1.8}.  The details are left to the interested reader. 
\end{rem}

\begin{rem}
Recently, the Main Conjecture has also been established for affine subspaces of $\mathbb{R}^n$ satisfying a certain diophantine condition by Jason Liu and the author \cite{HLiu}.  A variant of the counting problem considered in this section, i.e.  counting the number of rational points with a fixed denominator lying close to a space curve in $\mathbb{R}^3$ with non-vanishing torsion, has been partially solved by the author \cite{hua4}. 
\end{rem}

\section{A diophantine inequality of Robert and Sargos}\label{s2}
In \cite{RS}, Robert and Sargos obtained an essentially sharp upper bound on the number of integral solutions $(m_1,m_2,m_3,m_4)\in [M+1, 2M]^4$ to the diophantine inequality\footnote{Here we have taken the liberty of rewriting their result in our language, as their $\delta$ is actually $M^{-1}$ times our $\delta$.}
$$
\left|m_1^\alpha+m_2^\alpha-m_3^\alpha-m_4^\alpha\right|\le \delta M^{\alpha-1},
$$
where $\alpha\in\R\backslash\{0,1\}$, $M\ge2$ and $\delta>0$.
After de-homogenization, this is of course equivalent to counting rational points close to the surface $\textrm{RS}:  x_1^\alpha+x_2^\alpha-1=y^\alpha$, where $(x_1,x_2)\in [1,2]^2$. Indeed, their result, in our language, states that 
$$
N_\textrm{RS}(M,\delta)\ll_\varepsilon \delta M^{3+\varepsilon}+M^{2+\varepsilon}.
$$

It is readily verified that the surface RS admits nonvanishing Gaussian curvature, therefore our Theorem \ref{t2} implies that for some $c>0$
$$
N_\textrm{RS}(M,\delta)\ll \delta M^{3}+M^{2}\exp(c\sqrt{\log M}),
$$
which is slightly better. Our Theorem \ref{t1} also gives the lower bound  $N_\textrm{RS}(M,\delta)\gg  \delta M^{3}$ when $\delta\gg M^{-1+\varepsilon}$. 

The core of the proof in \cite{RS} relies on a recurrence formula based also on the van der Corput $B$-process and the positivity, which is somewhat reminiscent of our bootstrap argument. However, it is worth noting that they also crucially exploit the symmetry of the underlying manifold, which is absent in our general setting. In fact, their Lemma 1 immediately translates the counting problem above into one about the mean value of the fourth moment of some one dimensional exponential sum with a monomial phase. Then the well known one dimensional van der Corput $B$-process can be readily applied, whereas our exponential sum and its analysis are always genuinely multi-dimensional. Also, because of the one dimensional nature of their proof, the notion of a dual hypersurface does not seem to appear explicitly in \cite{RS}.

\section{Serre's Dimension Growth Conjecture}\label{s3}

For any $n\ge 2$, let $V\subset \mathbb{P}^n$ be an irreducible variety of degree $d\ge1$ defined over $\mathbb{Q}$. Let $x=[\mathbf{x}]\in\mathbb{P}^n(\mathbb{Q})$ be a projective rational point, with $\mathbf{x}\in \mathbb{Z}^{n+1}$ chosen so that $\gcd(x_0,x_1,\ldots,x_n)=1$. Then we define the \emph{height} of $x$ to be
$$H(x):=\max_{0\le i\le n}|x_i|$$
and the counting function associated with $V$ to be
$$N_V(B):=\#\{x\in V(\mathbb{Q}):H(x)\le B\}.$$
All known examples of asymptotic formulae for the counting function $N_V(B)$ take the shape
$$N_V(B)\sim cB^a(\log B)^b,$$
as $B\rightarrow\infty$, for $a,b,c\ge0$ such that $a\in\mathbb{Q}$ and $b\in\frac12\mathbb{Z}$. Also, there is a deep conjecture of Batyrev and Manin \cite[Conjecture C']{BM}, which connects the counting function $N_U(B)$ for a proper Zariski open subset $U$ of a Fano variety $V$ to some geometric invariants of the underlying variety $V$.

For an arbitrary projective variety $V$, we may still want to know whether anything meaningful can be said about the corresponding counting function $N_V(B)$. Nontrivial lower bounds in this generality are impossible since $V$ may not contain any rational points at all. On the other hand, it is easy to see that $N_V(B)\ll_{d,n}B^{\dim V+1}$ where the $\ll$ sign can be reversed when $d=1$. Serre \cite{Ser} conjectured that
\begin{conj}\label{c1}[The Dimension Growth Conjecture]
Let $V\subset\mathbb{P}^n$ be an irreducible projective variety of degree $d\ge2$ defined over $\mathbb{Q}$. Then 
$$N_V(B)\ll_V B^{\dim V}(\log B)^c$$ for some constant $c=c_V>0$.
\end{conj}

We note that the upper bound in Conjecture \ref{c1} is best possible up to powers of $\log B$, in that, when $V$ contains a linear divisor defined over $\mathbb{Q}$, we have $N_V(B)\gg_V B^{\dim V}$. But there exist examples, such as $x_1x_2=x_3x_4$, to show that the logarithmic factors cannot be completely dispensed with in general. 

By a projection argument, Conjecture \ref{c1} can be reduced to the case that $V$ is a hypersurface. A variety of methods have since been developed to attack this conjecture, most notably including the large sieve, geometry of numbers and the $q$-analogue of the van der Corput method. But it was not until Heath-Brown systematically developed the \emph{determinant method} of Bombieri and Pila \cite{BP} in a spectacular paper \cite{HB2} that significant progress could be made towards the full solution of this conjecture. The subsequent development along this line culminates in a series of papers by Browning, Heath-Brown and Salberger (see \cite{Bro2, BHS} for a detailed list of references), in which a slightly weaker form of Conjecture \ref{c1}, that $N_V(B)\ll_{V,\varepsilon} B^{\dim V+\varepsilon}$ for any $\varepsilon>0$, has been completely proved. There is also the \emph{Uniform Dimension Growth Conjecture} in which the implied constant is only allowed to depend on the ambient dimension and the degree of the variety. This conjecture has been solved for varieties of degree four and higher, and remains open for the case of cubic hypersurfaces.  

Here, we would like to address an observation that our Theorem \ref{t2} has a direct corollary that can be naturally regarded as the analogue of the Dimension Growth Conjecture in the context of smooth manifolds. 

Let $X$ be a compact submanifold of $\mathbb{R}^n$ and $(\mathbf{x},\mathbf{f}(\mathbf{x}))$ be a local chart of $X$, where 
$\mathbf{f}=(f_1,\ldots,f_k)$ and $\mathbf{x}=(x_1,\ldots,x_m)$.  We say $X$ satisfies \emph{property P} at $\mathbf{x}_0$ if  there exists a unit vector $\mathbf{u}\in\R^k$ such that
$$
\det\left\{\frac{\partial^2(\mathbf{u}\cdot\mathbf{f})}{\partial x_i\partial x_j}(\mathbf{x}_0)\right\}_{m\times m}\neq 0.
$$ 
$X$ is said to satisfy Property P if it satisfies property P everywhere. 

The analogous counting function of $N_V(B)$ for projective varieties here for the manifold $X$ is 

$$N_X(B):=\#\left\{\frac{\mathbf{a}}q\in X(\mathbb{Q}):\mathbf{a}\in\mathbb{Z}^n, q\le B\right\}.$$

Then by letting $\delta=0$ in Theorem \ref{t2}, it is not hard to deduce the following theorem.

\begin{thm} \label{t4}
Let $X$ be a compact submanifold of $\mathbb{R}^n$ which satisfies property P. Then 
\[
N_{X}(B)\ll
\left\{
\begin{array}{ll}
B\log B,&\text{when }\dim X=1\\B^{2}\exp(c\sqrt{\log B}),&\text{when }\dim X=2\\
 B^{\dim X}(\log B)^\kappa, &\text{when }\dim X\ge 3
\end{array}
\right.
\]
where  $c, \kappa$ and the implicit constant depend only on ${X}$.
\end{thm}

\begin{rem}
We note that property P is much weaker than the usual notion of non-degeneracy. Manifolds contained in a hyperplane are known to be degenerate, but these manifolds may still possess property P. Moreover, in the special case that $X$ is a nonsingular projective variety in $\mathbb{P}^n$ defined by $F(x_0, \ldots, x_n)=0$ for some homogeneous polynomial $F$, we may apply Theorem \ref{t4} to count the number of integral points on $X$ lying in homothetic dilations $B\mathcal{W}$ of a fixed region $\mathcal{W}$ bounded away from the origin. We note that Property P, in this particular case, is equivalent to the statement that the Hessian determinant of $F$ is nonvanishing on $\mathcal{W}$, cf. \cite[\S 1.1.4]{Do} and \cite{Go} for the details.
\end{rem}

The reason this theorem works for manifolds of arbitrary dimension in $\mathbb{R}^n$ is that we can use birational projection to reduce the problem to the case of hypersurfaces, as had been done in the original approach to the Dimension Growth Conjecture. The proof of Theorem \ref{t4} modulo Theorem \ref{t2} is given in \S \ref{s5}.

Comparing Theorem \ref{t4} with Conjecture \ref{c1}, we notice the following interesting features. The manifold $X$ in Theorem \ref{t4} need not be algebraic; even in the case of varieties, $X$ need not be defined over a number field. So in this sense, Theorem \ref{t4} addresses objects more general than does Conjecture \ref{c1}. On the other hand, to achieve this generality in Theorem \ref{t4}, we have to assume property P, which can be compared with the degree assumption in Conjecture \ref{c1}. It remains an interesting question whether Theorem \ref{t4} can be adapted to give an alternative proof of Conjecture \ref{c1}.

\section{Diophantine approximation on manifolds}\label{s4}

Let $\psi:\mathbb{N}\rightarrow\mathbb{R}^+$ be decreasing, which from now on we will call an \emph{approximation function}. We denote
\begin{equation*}
\mathscr{S}_n(\psi):=\{\mathbf{x}\in\mathbb{R}^n:\exists^\infty{q}\in\mathbb{Z}\textrm{ such that } \max_{1\le i\le n}\|{q}{x_i}\|<\psi(|{q}|)\}
\end{equation*}

and
\begin{equation*}
\mathscr{A}_n(\psi):=\{\mathbf{x}\in\mathbb{R}^n:\exists^\infty \mathbf{q}\in\mathbb{Z}^n\textrm{ such that } \|\mathbf{q}\cdot\mathbf{x}\|<\psi(|\mathbf{q}|_\infty)\},
\end{equation*}
where $\|\cdot\|$ measures the distance to the nearest integer and $\exists^\infty$ means ``there exist infinitely many".

The points in the sets $\mathscr{S}_n(\psi)$ and $\mathscr{A}_n(\psi)$ are said to be \emph{simultaneously $\psi$-approximable} and \emph{dually $\psi$-approximable} respectively.
The classical theory of metric diophantine approximation studies the measure theoretic properties of these sets in terms of the convergence of certain volume sums involving the approximation function $\psi$, which is known as the \emph{Khintchine-Groshev theory}. The fundamental principle that underlies the theory is a beautiful \emph{Zero-versus-One Dichotomy}---the sets in question (e.g. $\mathscr{A}_n(\psi)$ and $\mathscr{S}_n(\psi)$) have either zero or full Lebesgue/Hausdorff measure. Modern developments reveal that the problem becomes much deeper when the underlying space of $\mathbf{x}$ is restricted to a proper `curved' submanifold of $\mathbb{R}^n$. The following conjectures, which lie on the central stage, are generally known as the \emph{Generalized Baker-Schmidt Problems}.

\begin{conj}[]\label{c3}
For any approximation function $\psi$, any $s>m-1$ and any submanifold $\mathcal{M}$ of $\mathbb{R}^n$ of dimension $m$ which is non-degenerate everywhere except possibly on a set of zero Hausdorff $s$-measure, one has
\begin{equation*}
\mathcal{H}^s(\mathscr{A}_n(\psi)\cap\mathcal{M})=\left\{
\begin{array}{lll}
0&\textnormal{if}& \displaystyle\sum_{q=1}^{\infty}\left(\frac{\psi(q)}{q}\right)^{s+1-m}q^n<\infty,\\
\mathcal{H}^s(\mathcal{M})&\textnormal{if}&
\displaystyle\sum_{q=1}^{\infty}\left(\frac{\psi(q)}{q}\right)^{s+1-m}q^n=\infty.
\end{array}
\right.
\end{equation*}
\end{conj}

Here $\mathcal{H}^s$ denotes the Hausdorff $s$-measure.

\begin{conj}\label{c5}
For any approximation function $\psi$, any $s>\frac{k}{k+1}m$ and any submanifold $\mathcal{M}$ of $\mathbb{R}^n$ of dimension $m$ which is non-degenerate everywhere except possibly on a set of zero Hausdorff $s$-measure, one has
\begin{equation*}
\mathcal{H}^s(\mathscr{S}_n(\psi)\cap\mathcal{M})=\left\{
\begin{array}{lll}
0&\textnormal{if}& \displaystyle\sum_{q=1}^{\infty}\left(\frac{\psi(q)}{q}\right)^{s+k}q^n<\infty,\\
\mathcal{H}^s(\mathcal{M})&\textnormal{if}&
\displaystyle\sum_{q=1}^{\infty}\left(\frac{\psi(q)}{q}\right)^{s+k}q^n=\infty.
\end{array}
\right.
\end{equation*}
\end{conj}

 To demonstrate the strength of these conjectures, we simply observe that  a landmark theorem of Kleinbock and Margulis \cite{KM}, stating that all non-degenerate manifolds are extremal, is a rather special case of these conjectures. More precisely, in the case $s=m$ and $\psi(q)=q^{-\nu}$, the above two conjectures are actually equivalent by Khintchine's transference principle and reduce exactly to the theory of extremal manifolds. 

Regarding the current state of play of Conjecture \ref{c3}, the divergence case has now been established completely by Beresnevich, Dickinson and Velani \cite{BDV1} using their ubiquity framework, improving on earlier work of Dickinson and Dodson \cite{DD}. However, the convergence case remained completely open until the author settled this problem for planar curves recently \cite{hua2}. 

On the other hand, Conjecture \ref{c5}, corresponding to the \emph{simultaneous approximation on manifolds}, is widely believed to be more difficult, as it is closely related to the counting problem in Section \ref{s1}. The divergence case in Conjecture \ref{c5}, assuming the analyticity of the manifold, is now a theorem of Beresnevich \cite[Theorem 2.5]{bere}, but the convergence case in general is wide open except for the case of planar curves \cite{VV}.
In fact, the upper bound \eqref{e1.3}, if true, would immediately imply the convergence case of Conjecture \ref{c5}. As a consequence of Theorem \ref{t2}, we are able to establish this for hypersurfaces. This represents the first complete Hausdorff theory for a general class of manifolds for Conjecture \ref{c5}, beyond the case of planar curves.

\begin{thm}\label{t5}
Let $n\ge 3$ be an integer. For any approximation function $\psi$, any $s>\frac{n-1}{2}$ and any $C^l$ hypersurface $\mathcal{S}\subseteq\R^n$ with nonvanishing Gaussian curvature everywhere except possibly on a set of zero Hausdorff $s$-measure, we have
$$
\mathcal{H}^s(\mathscr{S}_n(\psi)\cap\mathcal{S})=0\quad\textnormal{if}\quad \displaystyle\sum_{q=1}^{\infty}\left(\frac{\psi(q)}{q}\right)^{s+1}q^n<\infty.
$$ 
\end{thm}

A proof of Theorem \ref{t5} modulo Theorem \ref{t2} will be given in Section \ref{s9}. 

\begin{rem}

Recently, J. Liu and the author  \cite{HLiu} have established a stronger form of Conjecture \ref{c5} (with wider range of $s$) for affine subspaces of $\mathbb{R}^n$, subject to a natural diophantine condition on the matrix defining the affine subspace. 
\end{rem}

\section{The proof of Theorem \ref{t2}}\label{s7}
We recall that $w\in C_0^\infty(\R^{n-1})$ is a weight function  with $\mathrm{supp}(w)\subseteq\mathcal{D}$ and $N_{\su}^w(Q,\delta)$ is defined in \eqref{e1.11}. To prove Theorem \ref{t2}, it suffices to prove under the assumptions \eqref{e1.10} and \eqref{e1.12} that
\begin{equation}\label{e2.19}
N_{\su}^w(Q,\delta)\ll \delta Q^n+E_n(Q),
\end{equation}
with $E_n(Q)$ given in \eqref{e1.14} and \eqref{e1.15}. 
The general case will follow by an application of partitions of unity and the compactness of $\mathcal{S}$. 

Recall that
$$N_{\su}^w(Q,\delta)=\sum_{\substack{\mathbf{a}\in\Z^{n-1}\\q\le Q\\ \|qf(\mathbf{a}/q)\|<\delta}}w\left(\frac{\mathbf{a}}{q}\right),
$$
where $\|\cdot\|$ means the distance to the nearest integer. 

Let $\rchi_{\delta}(\theta)$ be the characteristic function of the set $\{\theta\in\R:\|\theta\|\le\delta\}$
and 
$$J=\left\lfloor\frac1{2\delta}\right\rfloor.$$
Then consider the Fej\'er kernel
$$\mathcal{F}_J(\theta)=J^{-2}\left|\sum_{j=1}^Je(j\theta)\right|^2=\left(\frac{\sin(\pi J \theta)}{J\sin(\pi\theta)}\right)^2=\sum_{j=-J}^J\frac{J-|j|}{J^2}e(j\theta),$$
where as usual $e(x):=e^{2\pi i x}$.
An easy computation shows that when $\|\theta\|\le\de$
$$\left(\frac{\sin(\pi J \theta)}{J\sin(\pi\theta)}\right)^2\ge\left(\frac{2\pi^{-1}\pi J \|\theta\|}{J\pi\|\theta\|}\right)^2=\frac4{\pi^2},$$
that is,
$$\rchi_\delta(\theta)\le \frac{\pi^2}4\mathcal{F}_J(\theta).$$

Next we insert the Fej\'er kernel and obtain
\begin{equation} \label{e2.1}
N_{\su}^w(Q,\delta)\le\frac{\pi^2}4\sum_{\substack{\mathbf{a}\in \Z^{n-1}\\q\le Q}}\sum_{j=-J}^J\frac{J-|j|}{J^2}w\left(\frac{\mathbf{a}}q\right)e\left(jqf\left(\frac{\mathbf{a}}{q}\right)\right).
\end{equation}
The term, with $j=0$ on the right of \eqref{e2.1}, contributes
\begin{equation}\label{e2.2}
\le C_0 \delta Q^n
\end{equation}
where $C_0$ only depends on the support and sup norm of $w$.

Now, without loss of generality, we will only need to focus on the sum over the range $0<j\le J$ in \eqref{e2.1} since the contribution from the other part $-J\le j<0$ is nothing but the complex conjugate of that from $0<j\le J$. Hence, we obtain, by the $n-1$ dimensional Poisson summation formula over $\mathbf{a}\in\Z^{n-1}$, that
\begin{align*}
&\sum_{\mathbf{a}\in \Z^{n-1}}w\left(\frac{\mathbf{a}}q\right)e\left(jqf\left(\frac{\mathbf{a}}{q}\right)\right)\\
=&
\sum_{\mathbf{k}\in\Z^{n-1}}\int_{\R^{n-1}}w(\mathbf{t}/q)e(jqf(\mathbf{t}/q)
-\mathbf{k}\cdot\mathbf{t})d\mathbf{t},
\end{align*}
which by the change of variable $\mathbf{t}=q\mathbf{x}$ is
\begin{equation*}
q^{n-1}\sum_{\mathbf{k}\in\Z^{n-1}}I(j,\mathbf{k};q)
\end{equation*}
where
\begin{equation*}
I(j,\mathbf{k};q):=\int_{\R^{n-1}}w(\mathbf{x})e(q(jf(\mathbf{x})-\mathbf{k}\cdot\mathbf{x}))d\mathbf{x}.
\end{equation*}
Thus it follows from \eqref{e2.1} that
\begin{equation}\label{e2.5}
N_{\su}^w(Q,\delta)\le C_0\delta Q^n+2\left|\sum_{q\le Q}q^{n-1}\sum_{j=1}^J\frac{J-j}{J^2}\sum_{\mathbf{k}\in\Z^{n-1}}I(j,\mathbf{k};q)\right|.
\end{equation}

The study of oscillatory integrals such as $I(j,\mathbf{k};q)$ represents a central topic in classical harmonic analysis and therefore there is an extremely rich literature on this subject (see Stein's monograph \cite[Chapter VIII]{st} for the relevant background). An essentially sharp pointwise upper bound for $I(j,\mathbf{k};q)$ is a direct consequence of the method of stationary phase. Unfortunately, this bound is insufficient for our purpose. We have to also exploit the cancellation that occurs in the new exponential sum resulting from the stationary phase. 

Let
$$U:=\{\mathbf{x}\in\R^{n-1}|w(\mathbf{x})\not=0\},\quad V:=\nabla f(U)$$
and
$$\mathcal{R}:=\nabla f(\mathcal{D}), \quad \rho:=\frac12\text{dist}(\partial\mathcal{R},\partial V), $$
where $\text{dist}(\cdot,\cdot)$ stands for the Euclidean distance on $\mathbb{R}^{n-1}$. 
Note that $U$ and $V$ are open sets with $\overline{U}\subseteq \mathcal{D}$ and $\overline{V}\subseteq \mathcal{R}$ and that $\rho>0$.

We first recall the concept of a \emph{dual hypersurface} $\mathcal{S}^*$ of $\mathcal{S}$. Let $f^*: \mathcal{R}\to\R$ be the function given by

\begin{equation}\label{e2.6}
f^*(\mathbf{y})=\mathbf{y}\cdot(\nabla f)^{-1}(\mathbf{y})-(f\circ(\nabla f)^{-1})(\mathbf{y}).
\end{equation}
Note that
\begin{align}
&\nabla f^*\nonumber\\=& (\nabla \mathbf{y})\cdot (\nabla f)^{-1}+\mathbf{y}\cdot\nabla(\nabla f)^{-1}\nonumber-(\nabla f\circ (\nabla f)^{-1})\cdot\nabla(\nabla f)^{-1}\\
=&(\nabla f)^{-1}\label{e2.20}
\end{align}
and that if $\mathbf{y}=\nabla f(\mathbf{x})$ in \eqref{e2.6} then
\begin{equation}\label{e2.4}
f^*(\mathbf{y})=\mathbf{x}\cdot \mathbf{y}-f(\mathbf{x}).
\end{equation}
From \eqref{e2.4}, it is immediately clear that $f^{**}=f$, which explains why $f^*$ is called the dual of $f$. Also when $f\in C^l(\mathcal{D})$, we know $f^*\in C^{l}(\mathcal{R})$.

Moreover, the Hessians of $f$ and $f^*$ satisfy the relation
$$
\nabla^2f^*(\mathbf{y})=(\nabla^2f(\mathbf{x}))^{-1}, \quad\text{where }\mathbf{y}=\nabla f(\mathbf{x}), 
$$
which, by \eqref{e1.10}, then implies
\begin{equation}\label{e2.21}
c_2^{-1}\le |\det\nabla^2f^*(\mathbf{y})|\le c_1^{-1}
\end{equation}
for all $\mathbf{y}\in \mathcal{R}$.

Now we are ready to evaluate the oscillatory integral $I(j,\mathbf{k};q)$. This can be effectively done via the Method of Stationary Phase (See L. H\"ormander \cite[Theorem 7.7.5]{ho} or E. Stein \cite[Chap. VIII, Prop. 6]{st}.

\begin{lem}[Non-stationary phase]\label{l1}
Let $u\in C_0^\infty(\R^d)$ and $\phi\in C^{l}(\mathbb{R}^d)$ with $\nabla \phi(\mathbf{x})\not=\mathbf{0}$ for $\mathbf{x}\in\mathrm{supp}(u)$. Then for any $\lambda>0$
$$\left|\int u(\mathbf{x})e(\lambda\phi(\mathbf{x}))d\mathbf{x}\right|\le C_l\lambda^{-l+1},$$
and furthermore $C_l$ depends only on upper bounds for finitely many derivatives of $u$ and $\phi$ and a lower bound for $|\nabla\phi|$ in the support of $u$.
\end{lem}

This is a simplification of Theorem 7.7.1 in \cite{ho} and can be proved simply by repeated integration by parts.

\begin{lem}[Stationary phase]\label{l2} Let $u\in C_0^\infty(\R^d)$ and $\phi\in C^{l}(\mathbb{R}^d)$ for some integer $l>\frac{d}2+4$. Suppose $\nabla\phi(\mathbf{x}_0)=\mathbf{0}$ and $\det\nabla^2\phi(\mathbf{x}_0)\not=0$. Let $\sigma$ be the signature of $\nabla^2\phi(\mathbf{x}_0)$ and $\Delta=|\det\nabla^2\phi(\mathbf{x}_0)|$. Suppose further $\nabla\phi(\mathbf{x})\not=\mathbf{0}$ for $\mathbf{x}\in\mathrm{supp}(u)\backslash\{\mathbf{x}_0\}$. Then for $\lambda>0$ we have
$$
\int u(\mathbf{x}) e(\lambda\phi(\mathbf{x}))d\mathbf{x}=e(\lambda\phi(\mathbf{x}_0)+\sigma/8)\Delta^{-\frac12}\lambda^{-\frac{d}2}\left(u(\mathbf{x}_0)+O(\lambda^{-1})\right)$$
where the implicit constant depends only on upper bounds for finitely many derivatives of $\phi$ and $u$ in the support of $u$ and a lower bound for $\Delta$.
\end{lem}

Lemma \ref{l2} essentially follows from \cite[Chap. VIII, Prop. 6]{st}\footnote{See also Tao's online notes at \url{https://www.math.ucla.edu/~tao/247b.1.07w/notes8.pdf} and in particular the remarks in the last paragraph on page 11.} except that we assume less differentiability on $\phi$ as remarked in \cite[p. 222]{ho}.

For fixed $j$, we split the set of $\mathbf{k}\in\Z^{n-1}$ into three disjoint subsets, each of which will be dealt with differently. Let
$$
\mathscr{K}_1=\{\mathbf{k}\in\Z^{n-1}\mid \mathbf{k}/j\in V\},
$$
$$
\mathscr{K}_2=\{\mathbf{k}\in\Z^{n-1}\mid \text{dist}(\mathbf{k}/j,V)\ge\rho\}
$$
and
$$
\mathscr{K}_3=\Z^{n-1}\backslash\mathscr{K}_1\cup\mathscr{K}_2.
$$
We denote by $N_i$ the contribution from those $\mathbf{k}\in\mathscr{K}_i$ to the sum in \eqref{e2.5}, for $i=1, 2, 3$. 

For $\mathbf{k}\in\mathscr{K}_2$, let  
$$\phi_1(\mathbf{x})=\frac{jf(\mathbf{x})-\mathbf{k}\cdot\mathbf{x}}{\text{dist}(\mathbf{k},jV)},\quad \lambda_1=q \cdot\text{dist}(\mathbf{k},jV).$$
Note that 
$$
|\nabla \phi_1(\mathbf{x})|=\frac{|j\nabla f(\mathbf{x})-\mathbf{k}|}{\text{dist}(\mathbf{k},jV)}\ge1
$$
since $\nabla f(\mathbf{x})\in V$ when $\mathbf{x}\in U$.

 We also need the fact that $\phi_1$ has uniformly bounded $C^l$ norm for all $j$ and $\mathbf{k}\in\mathscr{K}_2$. To see this, it suffices to verify that $\phi_1$ is uniformly bounded, since it is transparent that higher derivatives of $\phi_1$ are all uniformly bounded in view of the definition of $\mathscr{K}_2$. We further split $\mathscr{K}_2$ into two subsets.

If $|\mathbf{k}|\ge C_2 j$ for some constant $C_2$ so large that $\mathrm{dist}(C_2^{-1}V,\mathbf{0})<1/2$, then 
$$\mathrm{dist}\left(\frac{\mathbf{k}}{|\mathbf{k}|},\frac{j}{|\mathbf{k}|} V\right)\ge \mathrm{dist}(\mathbf{k}/|\mathbf{k}|,\mathbf{0})-\mathrm{dist}\left(\mathbf{0}, \frac{j}{|\mathbf{k}|}V\right)\ge\frac12$$ 
and 
$$\left|\frac{j}{|\mathbf{k}|}f(\mathbf{x})-\frac{\mathbf{k}}{|\mathbf{k}|}\cdot\mathbf{x}\right|\ll 1.$$ So clearly $|\phi_1|$ is uniformly bounded for all such $j$ and $\mathbf{k}$. On the other hand,
if $|\mathbf{k}|\le C_2j$, then 
$$
\frac{jf(\mathbf{x})-\mathbf{k}\cdot\mathbf{x}}{\text{dist}(\mathbf{k},jV)}\ll \frac{j+|\mathbf{k}|}{ j}\ll 1.
$$
In any case, 
$|\phi_1|$ is uniformly bounded for all $j$ and $\mathbf{k}\in\mathscr{K}_2$.

Then we apply Lemma \ref{l1} and obtain that
$$
I(j,\mathbf{k};q)=\int_{\R^{n-1}}w(\mathbf{x})e(\lambda_1\phi_1(\mathbf{x}))d\mathbf{x}\ll \lambda_1^{-l+1},
$$
when $\mathbf{k}\in \mathscr{K}_2$. This immediately implies that
\begin{equation}\label{e2.17}
\sum_{\mathbf{k}\in\mathscr{K}_2}I(j,\mathbf{k};q)\ll q^{-l+1}
\end{equation}
on noting that $l\ge n+1$ from \eqref{e1.16}.
Therefore,
\begin{equation}\label{e2.15}
N_2\ll \log Q.
\end{equation}

To treat the remaining cases, let $$\phi(\mathbf{x})=f(\mathbf{x})-\frac{\mathbf{k}}j\cdot\mathbf{x},\quad \lambda=q j.$$
We observe that for fixed $j$, each $\mathbf{k}\in j\mathcal{R}$ determines a unique \emph{critical point} $\mathbf{x}_{j,\mathbf{k}}$ such that 
$$\nabla\phi(\mathbf{x}_{j,\mathbf{k}})=\nabla f(\mathbf{x}_{j,\mathbf{k}})-j^{-1}\mathbf{k}=\mathbf{0}.$$
Clearly
$$
\mathbf{x}_{j,\mathbf{k}}=(\nabla f)^{-1}(\mathbf{k}/j)\in \mathcal{D}.
$$
Also note that by \eqref{e2.4}
$$
\phi(\mathbf{x}_{j,\mathbf{k}})=f(\mathbf{x}_{j,\mathbf{k}})-\frac{\mathbf{k}}j\cdot\mathbf{x}_{j,\mathbf{k}}=-f^*\left(\frac{\mathbf{k}}j\right).
$$
Moreover, as eigenvalues depend continuously on the matrix, the condition \eqref{e1.10} prevents eigenvalues of $ \nabla^2f(\mathbf{x})$ from changing signs for all $\mathbf{x}\in\mathcal{D}$, i.e. 
 the signature $\sigma$ remains the same for all $\nabla^2f(\mathbf{x}_{j,\mathbf{k}})$. Therefore we may apply Lemma \ref{l2} with $u(\mathbf{x})=w(\mathbf{x})$, $d=n-1$, and obtain that
\begin{equation}\label{e2.23}
I(j,\mathbf{k};q)=\frac{w(\mathbf{x}_{j,\mathbf{k}})}{\sqrt{|\det\nabla^2f(\mathbf{x}_{j,\mathbf{k}})|}}(qj)^{-\frac{n-1}2}e\left(-qjf^*\left(\frac{\mathbf{k}}{j}\right)+\frac{\sigma}8\right)+O\left((qj)^{-\frac{n+1}2}\right).
\end{equation}

Since $\mathscr{K}_3$ lies in a ball of diameter $\asymp j$, we have
$$
|\mathscr{K}_3|\ll j^{n-1}
$$
with the implied constant only depending on  $\mathcal{R}$. We also observe that for $\mathbf{k}\in\mathscr{K}_3$, $\mathbf{x}_{j,\mathbf{k}}\not\in U$, namely $w(\mathbf{x}_{j,\mathbf{k}})=0$. Thus we obtain from \eqref{e2.23} that
\begin{equation}\label{e2.18}
\sum_{\mathbf{k}\in\mathscr{K}_3}I(j,\mathbf{k};q)\ll|\mathscr{K}_3|(qj)^{-\frac{n+1}2}\ll q^{-\frac{n+1}2}j^{\frac{n-3}2}.
\end{equation}
Hence
\begin{equation}\label{e2.16}
N_3\ll \sum_{q\le Q}q^{n-1-\frac{n+1}2}J^{-1}\sum_{j\le J}j^{\frac{n-3}2}\ll\delta^{-\frac{n-3}2}Q^\frac{n-1}{2},
\end{equation}
on recalling that $J=\lfloor \frac1{2\delta}\rfloor$.

Lastly we treat the most difficult case $\mathbf{k}\in \mathscr{K}_1$. In this case, we have, by \eqref{e2.23} again, that
\begin{align}
&\sum_{q\le Q}q^{n-1}I(j,\mathbf{k};q)\nonumber\\
\ll& w^*(\mathbf{k}/j)j^{-\frac{n-1}2}Q^{\frac{n-1}2}\min\left(\|jf^*(\mathbf{k}/j)\|^{-1},Q\right)+j^{-\frac{n+1}2}Q^{\frac{n-1}2},\label{e2.7}
\end{align}
where
$$
w^*(\mathbf{y})=w((\nabla f)^{-1}(\mathbf{y})).
$$

\begin{lem}\label{l4} 
If for all $M\ge1$ and $T\ge1$
\begin{equation}\label{e2.12}
\sum_{\substack{t\le T\\\mathbf{s}\in tV\\\|tf^*(\mathbf{s}/t)\|<M^{-1}}}w^*\left(\frac{\mathbf{s}}{t}\right)\le A \left(M^{-1}T^n+T^{\beta^*}\log 2T\right)
\end{equation} 
holds for some $\beta^*\ge n-1$ and $A>0$, then we have for all $M\ge1$ and $T\ge1$
\begin{equation}\label{e2.13}
\sum_{\substack{t\le T\\\mathbf{s}\in tV\\\|tf^*(\mathbf{s}/t)\|<M^{-1}}}w^*\left(\frac{\mathbf{s}}{t}\right)t^{-\frac{n-1}2}\le 2A\left( M^{-1}T^{\frac{n+1}2}+T^{\beta^*-\frac{n-1}2}\log 2T\right) 
\end{equation}
and
\begin{equation}\label{e2.14}
\sum_{\substack{t\le T\\\mathbf{s}\in tV\\\|tf^*(\mathbf{s}/t)\|\ge M^{-1}}}w^*\left(\frac{\mathbf{s}}{t}\right)t^{-\frac{n-1}2}\left\|tf^*\left(\frac{\mathbf{s}}t\right)\right\|^{-1}\le 6A\left( T^{\frac{n+1}2}\log M+MT^{\beta^*-\frac{n-1}2}\log 2T\right).
\end{equation}

\end{lem}
\begin{proof}
By partial summation and \eqref{e2.12}, we have
\begin{align*}
&\sum_{\substack{t\le T\\\mathbf{s}\in tV\\\|tf^*(\mathbf{s}/t)\|<M^{-1}}}w^*\left(\frac{\mathbf{s}}{t}\right)t^{-\frac{n-1}2}\\
\le&A \left(M^{-1}T^n+T^{\beta^*}\log 2T\right)T^{-\frac{n-1}2}+\frac{n-1}2\int_1^TA \left(M^{-1}u^n+u^{\beta^*}\log 2u\right)u^{-\frac{n+1}2}du\\
\le&2A\left( M^{-1}T^{\frac{n+1}2}+T^{\beta^*-\frac{n-1}2}\log 2T\right). 
\end{align*}

 To prove \eqref{e2.14}, we cover the interval $[M^{-1},1/2]$ by the union of dyadic intervals $[2^{i-1}M^{-1},2^iM^{-1})$, $1\le i\le \log M/\log2$. It therefore follows that
\begin{align*}
&\sum_{\substack{t\le T\\\mathbf{s}\in tV\\\|tf^*(\mathbf{s}/t)\|\ge M^{-1}}}w^*\left(\frac{\mathbf{s}}{t}\right)t^{-\frac{n-1}2}\left\|tf^*\left(\frac{\mathbf{s}}t\right)\right\|^{-1}\\
\le&\sum_{i\le\log M/\log2}2^{1-i}M\sum_{\substack{t\le T\\\mathbf{s}\in tV\\\frac{2^{i-1}}{M}\le\left\|tf^*\left(\frac{\mathbf{s}}t\right)\right\|< \frac{2^i}M}}w^*\left(\frac{\mathbf{s}}{t}\right)t^{-\frac{n-1}2}\\
\overset{\eqref{e2.13}}{\le}&\sum_{i\le\log M/\log2}2A\left(2T^{\frac{n+1}2}+2^{1-i}MT^{\beta^*-\frac{n-1}2} \log 2T\right)\\
\le&6A\left(T^{\frac{n+1}2}\log M+MT^{\beta^*-\frac{n-1}2}\log 2T\right).
\end{align*}
\end{proof}

\begin{lem}\label{l5}
If for all $Q^*\ge1$ and $\delta^*\in(0,1/2]$
$$
N^{w^*}_{\mathcal{S}^*}(Q^*,\delta^*)\le A\left( 
\delta^*(Q^*)^{n}+(Q^*)^{\beta^*}\log 2Q^*\right)
$$ holds for some $\beta^*>n-1$ and $A\ge1$,
then for all $Q\ge1$ and $\delta\in(0,1/2]$
$$
N^w_{\mathcal{S}}(Q,\delta)\le C_0\delta Q^{n}+CA Q^{\beta}\log 2Q
$$
holds, with $C_0$ as in \eqref{e2.2}, for
$$
\beta:=n-\frac{n-1}{2\beta^*-n+1}.
$$
Here the positive constants $C_0$ and $C$ only depend on $\su$ and $w$. 
\end{lem}

\begin{proof} First of all, it follows from $\beta^*>n-1$ that $\beta>n-1$.
By \eqref{e2.7} and Lemma \ref{l4}, we have
\begin{equation*}
N_1\ll \left(A\left(J^{\frac{n+1}2}\log Q+QJ^{\beta^*-\frac{n-1}2}\log 2J\right)+\sum_{j=1}^J\sum_{\mathbf{k}\in \mathscr{K}_1}j^{-\frac{n+1}2}\right)\frac{Q^{\frac{n-1}2}}J,
\end{equation*}
where $J=\lfloor \frac1{2\delta}\rfloor$.
Therefore on noting that 
$$
|\mathscr{K}_1|\ll j^{n-1},
$$
we have
\begin{equation}\label{e2.22}
N_1\ll A\left(\delta^{-\frac{n-1}2}Q^{\frac{n-1}2}\log Q+\delta^{\frac{n+1}2-\beta^*}Q^{\frac{n+1}2}\log \frac1\delta\right).
\end{equation}

Now it follows from \eqref{e2.5}, \eqref{e2.15}, \eqref{e2.16} and \eqref{e2.22} that
\begin{equation}\label{e2.8}
N^w_{\mathcal{S}}(Q,\delta)\le C_0\delta Q^n+C_1A\left(\delta^{-\frac{n-1}2}Q^{\frac{n-1}2}\log Q+\delta^{\frac{n+1}2-\beta^*}Q^{\frac{n+1}2}\log \frac1\delta\right),
\end{equation}
where the positive constants $C_0$ and $C_1$ only depend on $\su$ and $w$.

Next notice that $N^w_{\mathcal{S}}(Q,\delta)$ is an increasing function of $\delta$ when $Q$ is fixed, so when $\delta\le Q^{\beta-n}$ we have
\begin{align*}
N^w_{\mathcal{S}}(Q,\delta)&\le N^w_{\mathcal{S}}(Q,Q^{\beta-n})\\
&\overset{\eqref{e2.8}}{\le}C_0 Q^\beta+C_1A\left(Q^{\frac{(n-1)(n-\beta)}2}Q^{\frac{n-1}2}\log Q+Q^{(n-\beta)(\beta^*-\frac{n+1}2)}Q^{\frac{n+1}2}\log Q^{n-\beta}\right)\\
&\le  (2C_0+C_1A)Q^{\beta}\log 2Q+C_1AQ^{\frac{n-1}2(n-\beta+1)}\log Q\\
&\le(2C_0+C_1A)Q^{\beta}\log 2Q+C_1AQ^{n-1}\log Q,
\end{align*}
where in the second last inequality we use the identity $(n-\beta)(\beta^*-\frac{n+1}2)+\frac{n+1}2=\beta$, and in the last inequality we apply the inequality $n-\beta+1<2$. Let $C=2(C_0+C_1)$. It then follows that when $\delta\le Q^{\beta-n}$ we have
$$N^w_{\mathcal{S}}(Q,\delta)\le CAQ^{\beta}\log 2Q.$$

On the other hand, when $\delta\ge Q^{\beta-n}$ we have
\begin{align*}
N^w_{\mathcal{S}}(Q,\delta)\overset{\eqref{e2.8}}{\le}&C_0\delta Q^n+C_1A\left(Q^{\frac{(n-1)(n-\beta)}2}Q^{\frac{n-1}2}\log Q+Q^{(n-\beta)(\beta^*-\frac{n+1}2)}Q^{\frac{n+1}2}\log Q^{n-\beta}\right)\\
\le & C_0\delta Q^{n}+C_1A\left(Q^{\beta}\log Q+Q^{n-1}\log Q\right)\\
\le &C_0\delta Q^{n}+CAQ^{\beta}\log 2Q,
\end{align*}
since in \eqref{e2.8} both  $\delta^{-\frac{n-1}2}$ and $\delta^{\frac{n+1}2-\beta^*}\log\frac1\delta$ are decreasing functions in $\delta$.

So, either way, the lemma follows.
\end{proof}

Now we are poised to prove \eqref{e2.19} and therefore complete the proof of Theorem \ref{t2}. Due to the projective duality, $f^{**}=f$, \eqref{e2.20} and \eqref{e2.21}, we may apply Lemma \ref{l5} successively. We simply observe that Lemma \ref{l5} also holds when the roles of $N_\su^w(Q,\delta)$ and $N_{\su^*}^{w^*}(Q^*,\delta^*)$ are switched, so we may adjust the constants $C_0$ and $C$ if necessary to make it work in both scenarios. 

To be more precise, we start with the trivial bound
$$
N^w_{\mathcal{S}}(Q,\delta)\ll Q^n,
$$
and bootstrap it by Lemma \ref{l5} repeatedly.  The first application yields the bound
$$
N^w_{\mathcal{S}}(Q,\delta)\ll \delta Q^n+Q^{n-1+\frac2{n+1}}\log Q,
$$
which already recovers (actually improves) the bound \eqref{e1.6}.
Further iteration yields even better bounds.
Let $\beta_1=n$
and
\begin{equation}\label{e2.9}
\beta_i=n-\frac{n-1}{2\beta_{i-1}-n+1}\quad\text{for }i\ge2.
\end{equation}
%If $Q^{\beta_i}>Q^{n-1}\log Q$, then we continue to apply Lemma \ref{l5}; otherwise, we terminate the iteration. This way, 
After iterating $i$ times, we arrive at a bound in the form
\begin{equation}\label{e2.10}
N^w_{\mathcal{S}}(Q,\delta)\ll \delta Q^n+C^iQ^{\beta_i}\log Q.
\end{equation}

From \eqref{e2.9}, we obtain
\begin{equation}\label{e2.11}
\beta_i-(n-1)=\frac{\beta_{i-1}-(n-1)}{\beta_{i-1}-\frac{n-1}2}.
\end{equation}
Note that when $n\ge 4$,
$$
\beta_{i-1}-\frac{n-1}2>\frac{n-1}2\ge\frac32.
$$
This shows that in general the sequence $\{\beta_i\}$ decreases exponentially to $n-1$ as $i$ increases. In this case we choose
$$
i=\left\lfloor\frac{\log\log Q}{\log (3/2)} \right\rfloor.
$$
Then \eqref{e2.10} becomes
$$
N^w_{\mathcal{S}}(Q,\delta)\ll \delta Q^n+Q^{n-1}(\log Q)^\kappa.
$$
The case $n=3$ is a little more tricky. Indeed the recursive relation \eqref{e2.11} in this case reads
$$
\beta_i-2=\frac{\beta_{i-1}-2}{\beta_{i-1}-1},
$$
which can also be written as
$$
\frac1{\beta_i-2}=1+\frac1{\beta_{i-1}-2}.
$$
Therefore
$$
\beta_i=2+\frac1i,
$$
which still converges to the expected value 2, but at a considerably slower rate. Now we need to optimize $C^iQ^{1/i}$, so the optimal choice is 
$$
i=\lfloor\sqrt{\log Q}\rfloor.
$$
Thus we obtain the bound
$$
N^w_{\mathcal{S}}(Q,\delta)\ll \delta Q^3+Q^{2}\exp(\tau\sqrt{\log Q}),
$$
where $\tau>0$ is a constant depending only on $\mathcal{S}$ and $w$.

\section{Proof of Theorem \ref{t1}}\label{s8}
We sketch the proof of Theorem \ref{t1} modulo Theorem \ref{t2}. A key step is the replacement of the Fej\'er kernel used in Section \ref{s7} by Selberg's magic functions. This way, there is no loss in the dealing of the main term and the asymptotic formula in Theorem \ref{t1} can therefore be established as the result of the duality argument by inserting \eqref{e2.19} as the input. This strategy was first used in \cite{hua1} to obtain the asymptotic formula \eqref{e1.8} for planar curves. 

First we recall some basic properties of Selberg's magic functions. See \cite[Chapter 1]{mo} for details about the construction of these functions.

Let $I=(\alpha, \beta)$ be an arc of $\R/\Z$ with $\alpha<\beta<\alpha+1$ and $\rchi_I(x)$ be its  characteristic function. Then there exist finite trigonometric polynomials of degree at most $J$
$$S^{\pm}_J(x)=\sum_{|j|\le J}\hat{S}^{\pm}_J(j)e(jx)$$ such that 
$$
S^{-}_J(x)\le \rchi_I(x)\le S^{+}_J(x)
$$
and
$$
\hat{S}^{\pm}_J(0)=\beta-\alpha\pm\frac1{J+1}
$$
and
$$
|\hat{S}^{\pm}_J(j)|\le\frac1{J+1}+\min\left(\beta-\alpha,\frac1{\pi|j|}\right)
$$
for $0<|j|\le J$.

Now let $\alpha=-\delta$ and $\beta=\delta$. Then singling out the term with $j=0$, we have
\begin{equation}\label{e6.1}
|N_\mathcal{S}^w(Q,\delta)-2\delta N_0|\le \frac{N_0}{J+1}+2\sum_{j=1}^Jb_j\left|\sum_{\substack{\mathbf{a}\in \Z^{n-1}\\ q\le Q}}w(\mathbf{a}/q)e(jqf(\mathbf{a}/q))\right|
\end{equation}
where
$$N_0=\sum_{\substack{\mathbf{a}\in \Z^{n-1}\\ q\le Q}}w(\mathbf{a}/q)$$
and
$$b_j=\frac1{J+1}+\min\left(\beta-\alpha,\frac1{\pi j}\right).$$

Here, it is worth noting that  $J$ is a free parameter at our disposal, but it will turn out at the end of the proof that a good choice is $J=Q$.

Now by the Poisson summation formula, we have
\begin{equation}\label{e6.2}
N_0=\sum_{q\le Q}\sum_{\mathbf{k\in\Z^{n-1}}}q^{n-1}\hat{w}(q\mathbf{k})=\frac{\hat{w}(0)}nQ^n+O(Q^{n-1})
\end{equation}
in view of the rapid decay of the Fourier transform $\hat{w}$.

Now to estimate the sum inside the absolute value on the right hand side of \eqref{e6.1}, we follow the same argument in Section \ref{s7} and obtain that for $0<j\le J$
\begin{align}
&\nonumber\sum_{\substack{\mathbf{a}\in \Z^{n-1}\\ q\le Q}}w\left(\frac{\mathbf{a}}q\right)e\left(jqf\left(\frac{\mathbf{a}}q\right)\right)\\ 
=&\nonumber\sum_{q\le Q}q^{n-1}\sum_{\mathbf{k}\in\Z^{n-1}}I(j,\mathbf{k};q)\\
\ll &\sum_{\mathbf{k}}w^{*}\left(\frac{\mathbf{k}}j\right)j^{-\frac{n-1}2}Q^{\frac{n-1}2}\min\left(\left\|jf^*\left(\frac{\mathbf{k}}j\right)\right\|^{-1}, Q\right)+j^{\frac{n-3}2}Q^{\frac{n-1}2}\label{e6.5}
\end{align}
where the last inequality follows from \eqref{e2.17}, \eqref{e2.18} and \eqref{e2.7}.

Now we need a slight variant of Lemma \ref{l4}.

\begin{lem}\label{l6} 
We have
\begin{equation*}
\sum_{\substack{j\le J\\\mathbf{k}\in jV\\\|jf^*(\mathbf{k}/j)\|<Q^{-1}}}b_j w^*(\mathbf{k}/j)j^{-\frac{n-1}2}\ll_{\mathcal{S}^*,w^*} Q^{-1}J^{\frac{n-1}2}+G_n(J) 
\end{equation*}
and
\begin{equation*}
\sum_{\substack{j\le J\\\mathbf{k}\in jV\\\|jf^*(\mathbf{k}/j)\|\ge Q^{-1}}}b_j w^*(\mathbf{k}/j)j^{-\frac{n-1}2}\|jf^*(\mathbf{k}/j)\|^{-1}\ll_{\mathcal{S}^*,w^*} J^{\frac{n-1}2}\log Q+G_n(J)Q,
\end{equation*}
where
$$
G_3(J)=\exp(c_1\sqrt{\log J})
$$
and
$$
G_n(J)=J^{\frac{n-3}2}(\log J)^\kappa,\quad n\ge4
$$
where $c_1$ and $\kappa$ only depend on $\mathcal{S^*}$ and $w^*$.\end{lem}

\begin{proof}
By applying Theorem \ref{t2} to $\mathcal{S}^*$, we have
\begin{equation}\label{e6.4}
\sum_{\substack{j\le J\\\mathbf{k}\in jV\\\|jf^*(\mathbf{k}/j)\|<Q^{-1}}}w^*(\mathbf{k}/j)\ll_{\mathcal{S}^*,w^*}Q^{-1}J^n+E_n(J),
\end{equation}
where
$$
E_3(J)=J^{2}\exp(c\sqrt{\log J})
$$
and
$$
E_n(J)=J^{n-1}(\log J)^\kappa,\quad n\ge4
$$
where $c$ and $\kappa$ only depend on $\mathcal{S^*}$ and $w^*$.\\

Now note that $b_j\ll 1/j$. Then the lemma follows from partial summation and \eqref{e6.4}, the same way as in the proof of Lemma \ref{l4}.
\end{proof}

By \eqref{e6.1}, \eqref{e6.2}, \eqref{e6.5} and Lemma \ref{l6}, we obtain that
\begin{align*}
|N_\mathcal{S}^w(Q,\delta)-\frac{2\hat{w}(0)}{n} \delta Q^n|\ll& \delta Q^{n-1}+Q^{n}/J+J^{\frac{n-1}2}Q^{\frac{n-1}2}\log Q+Q^{\frac{n+1}2}G_n(J)\\
\stackrel{J=Q}{\ll} &E_n(Q),
\end{align*}
which completes the proof of Theorem \ref{t1}.

\section{Proof of Theorem \ref{t3}}\label{s6}
Let $K={x}_0+\mathcal{B}$, where $\mathcal{B}$ is a compact convex set which contains the origin as an interior point. For convenience, we assume $x_0=0$; the general case can be proved similarly. Denote by ${\varrho}>0$ the radius of any inscribed ball $\varrho B^{n-1}\subseteq \mathcal{B}$ where $B^{n-1}$ is the $n-1$ dimensional unit ball. For any $\varepsilon>0$, let $\tau_{\pm}:=1\pm\varepsilon/\varrho$. Consider open balls centered at points in $\mathcal{B}$ with radii less than $\varepsilon$, such that $\nabla f$ is a diffeomorphism on each of the balls. By compactness, let $\{W_i\}$ be a finite open cover of $\mathcal{B}$ consisting of such open balls.  Likewise, let $\{ Y_j\}$ be such a finite open cover for $\tau_-\mathcal{B}$, where $\{Y_j\}$ consists of open balls centered at points in $\tau_-\mathcal{B}$. 

First of all, we recall the following property of convex sets \cite[p. 94, eq. (10)]{mu}. For $d>0$
\begin{equation}\label{e7.1}
y\in\mathcal{B}, \quad z\not\in(1+d)\mathcal{B}\quad\Rightarrow\quad |z-y|>\varrho d.
\end{equation}
We claim that 
$$
\tau_{-} \mathcal{B}\subseteq \cup Y_j\subseteq \mathcal{B}\subseteq \cup W_i\subseteq \tau_+\mathcal{B}.
$$
We only need to prove the second and the fourth inclusions. Let $z\not\in\mathcal{B}$ and $y=\tau_-y_0\in\tau_{-}\mathcal{B}$ with $y_0\in\mathcal{B}$; then by \eqref{e7.1}
$$
|z-y|=\tau_{-}|\tau_{-}^{-1}z-y_0|>\tau_{-} \rho(\tau_{-}^{-1}-1)=\varepsilon.
$$
Hence  $\cup Y_j\subseteq \mathcal{B}$. Similarly $ \cup W_i\subseteq \tau_+\mathcal{B}$.

Now let $\{w^+_i\}$ and $\{w^-_j\}$ be smooth partitions of unity of $\mathcal{B}$ and $\tau_{-}\mathcal{B}$ subordinate to $\{W_i\}$ and $\{ Y_j\}$ respectively. Thus
\begin{equation}\label{e7.2}
\rchi_{\tau_-\mathcal{B}}\le \sum w^-_j\le \rchi_{\mathcal{B}}\le \sum w^+_i\le \rchi_{\tau_+\mathcal{B}}
\end{equation}
and
$$
\sum_j N^{w^-_j}_\mathcal{S}(Q,\delta)\le N_\mathcal{S}^\mathcal{B}(Q,\delta)\le \sum_i N^{w^+_i}_\mathcal{S}(Q,\delta).
$$
We then apply Theorem \ref{t1} to each of $N^{w^-_j}_\mathcal{S}(Q,\delta)$ and $N^{w^+_i}_\mathcal{S}(Q,\delta)$, and obtain
$$
\sum_i N^{w^+_i}_\mathcal{S}(Q,\delta)=\sum_i\frac{2\hat{w}^+_i(0)}n\delta Q^n+O_{\varepsilon}(E_n(Q))
$$
and
$$
\sum_j N^{w^-_j}_\mathcal{S}(Q,\delta)=\sum_j\frac{2\hat{w}^-_j(0)}n\delta Q^n+O_{\varepsilon}(E_n(Q)).
$$
Note that by \eqref{e7.2}
$$
\sum_i \hat{w}^+_i(0)=\int\sum_i w^+_i\le |\tau_+\mathcal{B}|
$$
and
$$
\sum_j \hat{w}^-_j(0)=\int\sum_j w^-_j\ge |\tau_-\mathcal{B}|.
$$
Now the theorem follows on observing that
$$
|\tau_\pm\mathcal{B}|=|\mathcal{B}|+O(\varepsilon).
$$

\section{Proof of Theorem \ref{t4}}\label{s5}
It suffices to prove the theorem on a small open neighborhood of any point, then the general case follows by the compactness. For every $\mathbf{x}_0$ in a chart $(\mathbf{x},f(\mathbf{x}))$ there exists a unit vector $\mathbf{u}\in\R^k$ such that
$$
\det\left\{\frac{\partial^2(\mathbf{u}\cdot\mathbf{f})}{\partial x_i\partial x_j}(\mathbf{x}_0)\right\}_{m\times m}\neq 0.
$$ 
By the continuity of the above determinant with respect to $\mathbf{u}$ and the density of the rationals on the unit sphere,
we may find a unit rational vector $\frac{\mathbf{s}}r:=(\frac{s_1}r,\frac{s_2}r,\ldots,\frac{s_k}r)$, a constant $c>0$ and a neighborhood $U$ of $\mathbf{x}_0$ such that
\begin{equation}\label{e5.1}
\left|\det\left\{\frac{\partial^2(\frac{\mathbf{s}}r\cdot\mathbf{f})}{\partial x_i\partial x_j}(\mathbf{x})\right\}_{m\times m}\right|\ge c,\quad \mathbf{x}\in U.
\end{equation}

Consider the projection
$$
(x_1,x_2,\ldots, x_m, f_1, f_2, \ldots, f_k)\rightarrow\left(x_1,x_2,\ldots, x_m, r^{-1}(s_1f_1+s_2 f_2+\ldots+ s_kf_k)\right),
$$
which sends a rational point 
$$
\left(\frac{a_1}q,\frac{a_2}q,\ldots,\frac{a_m}q,\frac{b_1}q,\ldots,\frac{b_k}q\right)
$$
to
$$
\left(\frac{a_1}q,\frac{a_2}q,\ldots,\frac{a_m}q,\frac{\mathbf{s}\cdot\mathbf{b}}{qr}\right).
$$
Let $\mathcal{S}$ be the resulting hypersurface in $\mathbb{R}^{m+1}$, then $\mathcal{S}$ has Gaussian curvature bounded away from 0 in view of \eqref{e5.1}. So we clearly have
$$
N_X(B)\le N_\mathcal{S}(Br,0).
$$
Now Theorem \ref{t4} follows by applying Theorem \ref{t2} ($m\ge2$) or \eqref{e1.5} ($m=1$) to $\mathcal{S}$ and setting $\delta=0$.

\section{Proof of Theorem \ref{t5}}\label{s9}
The proof of Theorem \ref{t5}  modulo Theorem \ref{t2} is routine and well known. We include it here for completeness. Since Theorem \ref{t5} is a \emph{zero-versus-one} measure theoretic law, it suffices to prove it locally. By the implicit function theorem, a local chart of a $C^l$ hypersurface in $\R^n$ can be given by the Monge parametrization,
$$\mathcal{S}:=\{(x_1, x_2, \ldots, x_{n-1}, f(\mathbf{x})): \mathbf{x}=(x_1, x_2, \ldots, x_{n-1})\in I\}$$ 
where $I=[0,1]^{n-1}$ and $f\in C^l(I)$.
Let $\Omega_n(f, \psi)$ denote the projection of $\mathcal{S}\cap \mathscr{S}_n(\psi)$ onto $I$. Since this projection is bi-Lipschitz, we know that
$$\mathcal{H}^s(\mathcal{S}\cap \mathscr{S}_n(\psi))=0\quad\Longleftrightarrow \quad\mathcal{H}^s(\Omega_n(f, \psi))=0.$$
 Note that the set $B=\{\mathbf{x}\in I: \det \nabla^2f =0\}$ is closed and $\mathcal{H}^s(B)=0$ by assumption,  therefore $I\backslash B$ can be written as a countable union of closed hyperrectangles $I_i$ on which $f$ satisfies \eqref{e1.10}. The constants $c_1$ and $c_2$  associated with \eqref{e1.10} depend on the particular choices of $I_i$.  Now it suffices to prove that $\mathcal{H}^s(\Omega_n(f,\psi)\cap I_i)=0$ for all $i$, as this would imply that
\begin{align*} 
\mathcal{H}^s(\Omega_n(f,\psi))&\le \mathcal{H}^s(B\cup(\cup_{i=1}^\infty\Omega_n(f,\psi)\cap I_i))\\
&\le \mathcal{H}^s(B)+\sum_{i=1}^\infty\mathcal{H}^s(\Omega_n(f,\psi)\cap I_i)=0
\end{align*}
since $\mathcal{H}^s(B)=0$ by the hypothesis of Theorem \ref{t5} and therefore establish Theorem \ref{t5}. So without loss of generality, we may assume $f$ satisfies \eqref{e1.10} on $I$. Having set up the question in context, we now show that  for $s>\frac{n-1}2$ we have $\mathcal{H}^s(\Omega_n(f,\psi))=0$ under the condition 
\begin{equation}\label{e9.0}
\sum_{q=1}^\infty \psi(q)^{s+1}q^{n-1-s}<\infty.
\end{equation}

First of all, $\mathcal{H}^s(\Omega_n(f,\psi))=0$ automatically for all $s>n-1$. So we may assume that $s\le n-1$. Next choose $\eta>0$ such that $\eta<(2s-n+1)/(s+1)$. It is easily verified that there is no loss of generality in assuming that 
\begin{equation}\label{e9.2}
\psi(q)\ge q^{-1+\eta} \quad\text{for all }q,
\end{equation}
because if \eqref{e9.2} fails, one may replace $\psi$ with $\hat\psi(q):=\max(\psi(q), q^{-1+\eta})$ which satisfies \eqref{e9.2}.

We recall that $\Omega_n(f,\psi)$ consists of points $\mathbf{x}\in I$ such that
the system of inequalities 
\[
\left\{
\begin{array}{ll}
|x_i-\frac{a_i}q|<\frac{\psi(q)}q,&1\le i\le n-1\\
|f(\mathbf{x})-\frac{b}q|<\frac{\psi(q)}q,&
\end{array}
\right.
\]
is satisfied for infinitely many $(q, \mathbf{a}, b)\in\N\times\Z^{n-1}\times \Z$. It is without loss of generality to assume that $\frac{\mathbf{a}}q\in I$ for solutions of the above inequalities. For $\frac{\mathbf{p}}q\in\Q^n$ with $\mathbf{p}=(\mathbf{a},b)\in\Z^{n-1}\times\Z$, let $\sigma(\mathbf{p}/q)$ denote the set of $\mathbf{x}\in I$ satisfying the above inequality. Trivially we have
\begin{equation}\label{e9.1}
\text{diam}\left(\sigma\left(\frac{\textbf{p}}q\right)\right)\ll \frac{\psi(q)}q
\end{equation}
where the implied constant only depends on $n$.

Now assume that $\sigma(\mathbf{p}/q)\not=\emptyset$ and let $\textbf{x}\in\sigma(\mathbf{p}/q)$. Since $f\in C^l(I)$, $f$ is Lipschitz. It follows by the triangle inequality and the Lipschitz inequality that 
\begin{align*}
\left|f\left(\frac{\textbf{a}}q\right)-\frac{b}q\right|&\le\left|f(\mathbf{x})-\frac{b}q\right|+\left|f\left(\frac{\textbf{a}}q\right)-f(\mathbf{x})\right|\\
&\le \frac{\psi(q)}q+c_3\left|\mathbf{x}-\frac{\mathbf{a}}q\right|\\
&\le c_4\frac{\psi(q)}q,
\end{align*}
where $c_3$ and $c_4$ are constants that depend only on $f$. Thus for $i\ge0$
\begin{align*}
&\#\{\mathbf{p}/q\in\Q^{n}: 2^i\le q<2^{i+1}, \sigma(\mathbf{p}/q)\not=\emptyset\}\\
\le&\#\{\mathbf{p}/q\in\Q^{n}: 2^i\le q<2^{i+1}, \mathbf{a}/q\in I, |f(\mathbf{a}/q)-b/q|\le c_4\psi(q)/q\}\\
\le&\#\{\mathbf{a}/q\in\Q^{n-1}: q<2^{i+1}, \mathbf{a}/q\in I, \|qf(\mathbf{a}/q)\|\le c_4\psi(2^i)\}
\end{align*}
where in the last inequality the monotonicity assumption on $\psi$ is  used. Now in view of \eqref{e9.2}, Theorem \ref{t2} implies that
\begin{equation}\label{e9.3}
\#\{\mathbf{p}/q\in\Q^{n}: 2^i\le q<2^{i+1}, \sigma(\mathbf{p}/q)\not=\emptyset\}\ll \psi(2^i)2^{ni}.
\end{equation}

Recall that $\Omega_n(f,\psi)=\limsup\sigma(\mathbf{p}/q)$, which consists of $\mathbf{x}\in I$ lying in infinitely many $\sigma(\mathbf{p}/q)$. In view of \eqref{e9.1} and \eqref{e9.3}, we have
\begin{align*}
&\sum_{q=1}^\infty\sum_{\mathbf{p}\in\Z^{n}} \text{diam}\left(\sigma\left(\frac{\mathbf{p}}q\right)\right)^s\\
\ll& \sum_{i=0}^\infty \#\{\mathbf{p}/q\in\Q^{n}: 2^i\le q<2^{i+1}, \sigma(\mathbf{p}/q)\not=\emptyset\}\psi(2^i)^s/2^{is}\\
\ll &\sum_{i=0}^\infty \psi(2^i)^{1+s}2^{i(n-s)}\\
\ll &\sum_{q=1}^\infty \psi(q)^{1+s}q^{n-s-1}
\end{align*}
which by \eqref{e9.0} is convergent. Hence the Hausdorff-Cantelli Lemma \cite[p. 68]{BeD} implies that $\mathcal{H}^s(\Omega_n(f,\psi))=0$ and this completes the proof of Theorem \ref{t5}.

\bigskip

\proof[Acknowledgements]
The author is very grateful to the anonymous referees for providing  very detailed lists of comments, corrections and suggestions which greatly improved the presentation of this manuscript.

\end{document}